\documentclass[11pt, a4paper]{amsart}
\usepackage[english]{babel}
\usepackage[utf8]{inputenc}
\usepackage[T1]{fontenc}
\usepackage{amsfonts,amsmath,amssymb,amsthm,mathtools}
\usepackage[final]{graphicx}
\usepackage{tabularx}
\usepackage{subcaption}
\usepackage{comment}

\usepackage{color}
\usepackage{hyperref}
\usepackage{enumitem}
\usepackage{tikz}
\usepackage{tikz-cd}
\usetikzlibrary{arrows}
\usepackage{blkarray}
\usetikzlibrary{decorations.markings}

\newtheorem{theorem}{Theorem}[section]

\newtheorem{corollary}[theorem]{Corollary}
\newtheorem{proposition}[theorem]{Proposition}
\newtheorem{lemma}[theorem]{Lemma}

\newtheorem*{corollary*}{Corollary}

\theoremstyle{definition}
\newtheorem{definition}[theorem]{Definition}
\newtheorem{example}[theorem]{Example}

\theoremstyle{remark}
\newtheorem{remark}[theorem]{Remark}

\newcommand{\Nerve}{\mathrm{N}}

\newcommand{\bA}{\mathbf{A}}
\newcommand{\bB}{\mathbf{B}}
\newcommand{\bC}{\mathbf{C}}

\newcommand{\cR}{\mathcal{R}}

\newcommand{\Fin}{\mathbf{Fin}}
\newcommand{\Fun}{\mathbf{Fun}}

\newcommand{\Quiver}{\mathbf{Quiver}}
\newcommand{\Cat}{\mathbf{Cat}}
\newcommand{\Poset}{\mathbf{Poset}}
\newcommand{\Preord}{\mathbf{Preord}}
\newcommand{\Path}{\mathbf{Path}}
\newcommand{\Reach}{\mathbf{Reach}}
\newcommand{\sk}{\mathrm{sk \: }}
\newcommand{\Pre}{\mathbf{Pre}}

\newcommand{\C}{\mathbf C}
\newcommand{\D}{\mathbf D}

\newcommand{\ra}{\rightarrow}

\usepackage[margin=2.5cm]{geometry}

\title{On reachability categories, persistence, and commuting algebras of quivers}
\author{Luigi Caputi\(\dagger\)}

\author{Henri Riihim\"{a}ki\(^\ast\)}

\begin{document}

\maketitle
\begin{abstract}
    For a finite quiver~$Q$, we study the \emph{reachability category}~$\Reach_Q$. We investigate the properties of~$\Reach_Q$ from both a categorical and a topological viewpoint. In particular, we compare~$\Reach_Q$ with~$\Path_Q$,  the category freely  generated by $Q$. As a first application, we study the category algebra of $\Reach_Q$, which is isomorphic to the \emph{commuting algebra} of $Q$. As a consequence, we recover, in a categorical framework, previous results obtained by Green and Schroll; we show that the commuting algebra of $Q$ is Morita equivalent to the incidence algebra of a poset, the reachability poset. We further show that commuting algebras are Morita equivalent if and only if the reachability posets are isomorphic.  As a second application, we define \emph{persistent Hochschild homology of quivers} via reachability categories.
\end{abstract}


\section{Introduction}
To any small category \(\C\) we can associate a preorder structure on the set of objects \(\mathrm{Ob}(\C)\) by declaring \(x \leq y\) if and only if there is some morphism \(x \ra y\) in \(\C\). This association yields a functor $\Pre$ from the large category \(\Cat\) of small categories  into the category \(\Preord\) of preorders. The funtor $\Pre$, sometimes called \emph{preorder reflection}, has a right adjoint given by inclusion (every preordered set can be seen as a category). This is a foundational instance of an order structure approximation of categories; for a recent perspective on further order structures appearing in different mathematical domains, see also~\cite{Floystad_order_structures}.

Quivers, and morphisms of quivers, constitute the category \(\Quiver\). Then, the preorder reflection is part of a composition of adjunctions 
\begin{center}
    	\begin{tikzpicture}
    		\tikzstyle{point}=[circle,thick,draw=black,fill=black,inner sep=0pt,minimum width=2pt,minimum height=2pt]
    		\tikzstyle{arc}=[shorten >= 3pt,shorten <= 3pt,->, thick]
    		\tikzstyle{loop}=[in=-30,out=30,looseness=10]
    		\node[] (0) at  (0,0) {\(\Quiver\)};
    		\node[] (1) at  (2.5,0) {\(\Cat\)};
            \node[] (2) at  (5,0) {\(\Preord\)};
    		\node[] (path) at  (1.25,0.7) {\(\Path\)};
    		\node[] (U) at  (1.25,-0.7) {\(U\)};
            \node[] (pre) at  (3.75,0.7) {\(\Pre\)};
    		\node[] (inc) at  (3.75,-0.7) {\(i\)};
    		
    		\draw[arc, bend left = 30] (0) to (1);
    		\draw[arc, bend left = 30] (1) to (0);
            \draw[arc, bend left = 30] (1) to (2);
    		\draw[arc, bend left = 30] (2) to (1);
    	\end{tikzpicture}
    \end{center}
where the functor \(\Path\) maps a quiver to its \emph{path category} and \(U\) denotes the forgetful functor. There is a functor directly mapping a quiver \(Q\) to a preorder. We denote this functor by \(\Reach\) and the resulting preordered set, seen as a category, is the main object of study in this work; we call it the \emph{reachability category} of a quiver \(Q\) -- see Definition~\ref{def:reach_G}. In graph theory the notion of reachability  refers to the existence of a path between two vertices in a directed graph, and this gives rise to a preorder relation on the set of vertices. The reachability category is exactly the same notion from a categorical point of view. 

Concurrently to our paper, the notion of reachability (graph, poset, category) has appeared in various recent works concerning homotopy and homology theories of directed graphs \cite{Ivanov_Cayley_digraphs,hepworth2023reachability,ivanov2023nested, puca2023obstructions}, and especially in relation with the so called \emph{magnitude-path spectral sequence}~\cite{asao}. In fact, the homology of the (nerve of the) reachability category appears as the limit homology in the magnitude-path spectral sequence. We focus on the categorical and topological  properties of the reachability categories. Our main motivation stems from applications in persistent homology, and, in particular, from investigations of suitable \emph{persistent Hochschild homology theories of quivers}~\cite{persistentHH}. A main tool in the construction of persistent Hochschild homology is the \emph{condensation} of graphs, a standard operation which produces a directed acyclic graph by condensing the strongly connected components into vertices. 

The interest in reachability categories stems from the  observation that they are equivalent to what we call in this work \emph{reachability posets}. Categorically, this equivalence is obtained by passing to skeletal subcategories. This is the categorical analog of condensation, and from this viewpoint  condensation of quivers satisfies a universal property -- see Proposition~\ref{prop:condensuni}. Equivalence of categories also preserves Hochschild homology (of category algebras). Then, we use these properties to extend the persistence pipeline of \cite{persistentHH} to the whole category of quivers by defining persistent Hochschild homology of quivers to be Hochschild homology of the associated reachability categories; up to taking their ranks (i.e.~the Hochschild Betti curves of a filtration of quivers), this agrees with the Hochschild homology of the associated posets, giving computational benefits -- see Section \ref{subsec:homological_stuff}. We interpret the persistent Hochschild homology via the classical result of Gerstenhaber and Schack \cite{GERSTENHABER1983143}, and introduce what we call \emph{reachability persistent homology}.

A second application of reachability categories that we present in this work concerns Morita properties of path algebras. For a preorder, the equivalence of categories described above allows to construct a poset which has a Morita equivalent category algebra. Considering the path algebra of a quiver and of its associated reachability category, via this Morita equivalent construction we obtain a categorical enhancement of a former result of Green and Schroll~\cite{green2023commuting}: if  $Q$ is a finite quiver, and~$\mathbb{K}$ is a field, then  the quotient of the path algebra~$\mathbb{K}Q$ by its parallel ideal is Morita equivalent to an incidence algebra. We refer to Theorem~\ref{thm:enhance} for a more detailed description of this result. If $\cR$ denotes the functor that to a quiver associates the reachability poset, restricted to the category $\Quiver_0$ of finite acyclic quivers, it follows that the diagram
 \begin{center}
 \begin{tikzcd}
\Quiver_0\arrow[dr,"\mathcal{C}"'] \arrow[r, "\cR"] & \Poset \arrow[d, "I"]\\
  & \mathbf{Alg}_{\mathbb{K}}
 \end{tikzcd}    
 \end{center}
is commutative up to Morita equivalences (here $I$ denotes the incidence algebra of a poset and $\mathcal{C}$ the commuting algebra). Then, using a classical result of Stanley~\cite[Theorem~1]{stanley}, we also show that given  commuting algebras $\mathcal{C}(Q)$ and $\mathcal{C}(Q')$ are Morita equivalent if and only if the reachability posets $\cR(Q)$ and $\cR(Q')$ are isomorphic (Theorem~\ref{thm:stley}). Considering the commutative diagram above, as an outlook, it would be interesting to see whether the category $\Quiver$ admits a Quillen model structure where the weak equivalences are given by the Morita equivalences.   

\subsection*{Acknowledgements}
The authors wish to thank Ehud Meir for his useful comments and feedbacks on the first draft of the paper. LC wishes  to warmly thank Francesco Vaccarino for valuable discussions on the topic. LC is grateful to INdAM-GNSAGA for its partial support on the research activity. HR acknowledges funding from the The Wallenberg Initiative on Networks and Quantum Information, and from the Dbrain project within Digital Futures consortium at KTH Royal Institute of Technology

\section{Quivers and categories}
In this section, we collect some basic notions about quivers, and recall the main definitions needed in the follow-up. Let \(\mathbf{2}\) denote the category with objects~\(E\) and $V$, and two non-identity morphisms \(s,t \colon E \ra V\), called \emph{source} and \emph{target}. Let \(\Fin\) be the full subcategory of $\mathbf{Set}$ of finite sets. Then, a (finite) \emph{quiver} is a functor \(Q \colon \mathbf{2} \ra \Fin\). Equivalently, a finite quiver can be represented as a directed graph with a set of vertices~$V$ and a set of directed edges~$E$. For each edge $e\in E$ the source and target maps describe the source $s(e)$ and the target $t(e)$ of~$e$; we will graphically represent an edge $e$ by an arrow $s(e)\longrightarrow t(e)$. When using this representation, we will also denote a quiver with the quadruple $(V,E,s,t)$. We sometimes denote edges~$e$ by ordered pairs~\((v,w)\) of vertices corresponding to the source~$v$ and the target~$w$ of $e$. Note that \emph{loops}, i.e.\ edges of type~\((v,v)\), and multiple edges between two vertices are allowed. Morphisms of quivers are natural transformations of functors. The category~$\Quiver$ of finite quivers and morphisms of quivers is the functor category $\Fun(\mathbf{2},\Fin)$.

\begin{remark}\label{rem:quiver_morphisms}
	Let \(Q=(V,E,s,t)\) and \(Q'=(V',E',s',t')\) be two quivers. A morphism \(f \colon Q \ra Q'\) boils down to requiring that the two  diagrams 
	\begin{center}
	\begin{tikzcd}
		E\arrow[d,"f_E"'] \arrow[r, "s"] & V \arrow[d,"f_V"]\\
		E' \arrow[r, "s'"] & V'
	\end{tikzcd}
 \quad
 and \quad
	\begin{tikzcd}
		E\arrow[d,"f_E"'] \arrow[r, "t"] & V \arrow[d,"f_V"]\\
		E' \arrow[r, "t'"] & V'
	\end{tikzcd}  
	\end{center}	
 commute. Note that morphisms of quivers can collapse edges to loops. Consider for example the following quivers:
    
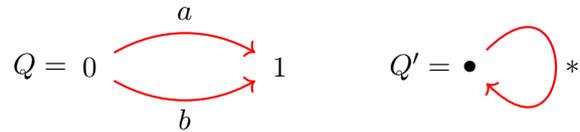
\begin{center}
    	\begin{tikzpicture}
    		\tikzstyle{point}=[circle,thick,draw=black,fill=black,inner sep=0pt,minimum width=2pt,minimum height=2pt]
    		\tikzstyle{arc}=[shorten >= 3pt,shorten <= 3pt,->, thick]
    		\tikzstyle{loop}=[in=-45,out=45,looseness=10]
    		\node[] (0) at  (0,0) {0};
    		\node[] (1) at  (2.5,0) {1};
    		\node[] (a) at  (1.25,0.7) {\(a\)};
    		\node[] (b) at  (1.25,-0.7) {\(b\)};
    		
    		\draw[arc,red, bend left = 30] (0) to (1);
    		\draw[arc,red, bend right = 30] (0) to (1);
    		
    		\node at (-0.65,0) {\(Q=\)};
    		
    		\node[] (0') at (5,0) {\(\bullet\)};
    		\node[] (e) at  (6.35,0) {\(\ast\)};   		
    		\path[->,thick] (0') edge[loop right, red] (0');
    		\node at (4.35,0) {\(Q'=\)};
    	\end{tikzpicture}
    \end{center}
	Then, the rules \(f_E(a) = f_E(b) = \ast\), and \(f_V(0) = f_V(1) = \bullet\) describe a morphism of   quivers. 
 \end{remark}

A \emph{path} from a vertex \(v\) to a vertex \(w\) in a quiver is a sequence~\((e_0,\dots,e_n)\) of edges  such that \(s(e_0)=v\), \(t(e_n)=w\), and \(t(e_i)=s(e_{i+1})\). A \emph{simple} path is a path in which the edges contain no vertex that is repeated twice.

\begin{figure}[h]
	\begin{tikzpicture}[baseline=(current bounding box.center)]
		\tikzstyle{point}=[circle,thick,draw=black,fill=black,inner sep=0pt,minimum width=2pt,minimum height=2pt]
		\tikzstyle{arc}=[shorten >= 8pt,shorten <= 8pt,->, thick]
		
		\node[above] (v0) at (0,0) {$v_0$};
		\draw[fill] (0,0)  circle (.05);
		\node[above] (v1) at (1.5,0) {$v_1$};
		\draw[fill] (1.5,0)  circle (.05);
		\node[] at (3,0) {\dots};
		\node[above] (v4) at (4.5,0) {$v_{n-1}$};
		\draw[fill] (4.5,0)  circle (.05);
		\node[above] (v5) at (6,0) {$v_{n}$};
		\draw[fill] (6,0)  circle (.05);
		
		\draw[thick, red, -latex] (0.15,0) -- (1.35,0);
		\draw[thick, red, -latex] (1.65,0) -- (2.5,0);
		\draw[thick, red, -latex] (3.4,0) -- (4.35,0);
		\draw[thick, red, -latex] (4.65,0) -- (5.85,0);
	\end{tikzpicture}
	\caption{The linear quiver~$I_n$.}
	\label{fig:nstep}
\end{figure}
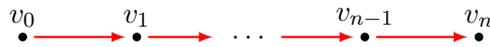

Note that the definition of a path allows  loops as well. The morphism \(f\) described in Remark~\ref{rem:quiver_morphisms} sends both paths \(a\) and \(b\) in \(Q\) to the same path in \(Q'\): the path containing the loop~\(\ast\). On the other hand, in the sequence of edges describing a simple path, loops are not allowed. Saying otherwise, a simple path is an isomorphic copy of the quiver $I_n$ illustrated in Figure~\ref{fig:nstep}. We call \emph{length} of a (simple) path the number of its edges; e.g.\ the length of $(e_0,\dots,e_n)$ is $n+1$. The following is straightforward from the definitions:

\begin{lemma}\label{lem:paths_to_paths}
    Morphisms in $\Quiver$ send paths to paths. 
\end{lemma}

We can associate to any quiver \(Q\) a small category~$\Path_{Q}$, called the  \emph{path category} of~$Q$. This is the category freely generated by $Q$. Spelling it out, the path category has the vertices of~\(Q\) as objects. The set of morphisms between the  vertices \(v\) and \(w\) consists of all possible paths in~$Q$ from \(v\) to \(w\). For each vertex \(v\) the trivial path with an empty sequence of edges is taken to be the identity morphism \(1_v\) at \(v\).  There is a forgetful functor~$U$ from the category~$\Cat$ of small categories and functors to the category of (possibly infinite) quivers, obtained by forgetting which arrows are the identities and which the compositions. Such  forgetful functor has a left adjoint, which is the free functor sending a quiver \(Q\) to~$\Path_{Q}$ (see e.g.\ \cite[Section~II.7]{maclane:71}). 

In some cases, it might be difficult to directly work with the category $\Path_{Q}$. For example, if the quiver~\(Q\) contains directed cycles, i.e.\ paths from a vertex \(v\) to itself, then the free functor will enforce the category~$\Path_{Q}$ to have infinitely many morphisms. Furthermore, category algebras of categories with directed cycles are inifinite dimensional. As our aim is to work with general quivers, not necessarily acyclic, in Section~\ref{sec:reach} we will introduce  another family of categories naturally associated to quivers, which is the main object of study in this paper.

We now proceed with recalling the notion of category algebras.

\begin{definition}
    Let $\bC$ be a category and $R$ be a commutative ring with unity. The \emph{category algebra} $R\bC$ is the free $R$-module with basis the set of morphisms of $\bC$. The product on the basis elements is given by
    \[f \cdot g = 
    \begin{cases}
        f \circ g & \text{when the composition exists in $\bC$}\\
        0 & \text{otherwise}    
    \end{cases}\]
    and then it is linearly extended to the whole $R\bC$. 
\end{definition}

The category algebra $R\bC$ is an associative \(R\)-algebra. If $\bC$ has finitely many objects, then~$R\bC$ is also unital. The unit is given by $\sum_{c\in \bC} 1_c$, where $1_c$ is the identity endomorphism of the object~$c$ in $\bC$. The definition of the category algebra is a generalisation of the classical definition of group algebra. In fact, if $G$ is a group, seen as a category with a single object and $G$ as  morphisms, then the associated category algebra is  the classical group algebra. We also have the following classical examples. 

\begin{example}\label{ex:path algebras}
    When $\bC$ is the path category $\Path_Q$, then the category algebra $R\Path_Q$ is the classical \emph{path algebra} of \(Q\).
\end{example} 

Recall that every poset~$(P,\leq)$ can be seen as a category~$\mathbf{P}$ in a standard way: there is a unique morphism $p\to q$ if and only if $p\leq q$. For a poset $P$, Rota introduced the definition of an \emph{incidence algebra}, cf.\ \cite[Section~3]{rota}; this is the algebra generated by the relations $p\leq q$, with convolution product. Equivalently, the incidence algebra of a poset is the quotient of its path algebra with respect to the \emph{parallel ideal}, i.e.~the two-sided ideal generated by all differences of paths with the same source and the same end vertices~\cite{MR982636}. Incidence algebras provide other examples of category algebras:
 
\begin{example}
Let $P$ be a finite poset and $\mathbf{P}$ its associated category. Then, the category algebra~\(R \mathbf{P}\) is isomorphic to the incidence algebra of \(P\). 
\end{example}

A category~$\bC$, via the forgetful functor $U$,  can be regarded also as a quiver; hence, we can form the path category $\Path_\bC$.

\begin{remark}\label{rem:fajah}
By \cite[Proposition 2.2.6]{Xu_thesis}, the obvious functor \(\phi \colon \Path_\bC \ra \bC\) induces a surjective homomorphism \(\phi \colon R\Path_\bC \ra R\bC\); its kernel is  generated by \[\{\xrightarrow{\alpha_1}\xrightarrow{\alpha_2}-\xrightarrow{\alpha_2\alpha_1}\}\] where \(\alpha_1\) and \(\alpha_2\) are morphisms of $\bC$. Then this map induces a natural isomorphism of \(R\)-algebras between \(R\Path_\bC / \ker(\phi)\) and \(R\bC\).
\end{remark}

When $\bC$ is a poset~$P$ (seen as a category), then the map $\phi$ in Remark~\ref{rem:fajah} has the path algebra of $P$ as domain, and the incidence algebra of $P$ as target. Hence, the induced isomorphism yields the equivalent definition of incidence algebras as quotient of path algebras  (by the parallel ideal). If the base ring is a field, the parallel ideal is zero if and only if $P$ is a tree (as a poset, i.e.\ if for each $p \in P$, the set $\{ s \in P \mid  s < p\}$  is well-ordered); see also~\cite{ORTEGA2006225}. This happens if and only if the incidence algebra of $P$ is hereditary.  We can summarise it as follows:

\begin{remark}\label{rem:pathalgincalg}
    Let $P$ be a finite poset. Then its associated path algebra and incidence algebra are isomorphic if and only if $P$ is a tree.
\end{remark}

We conclude the section with recalling the definition of quiver condensation. A quiver is \emph{strongly connected} if it contains a path from $x$ to $y$ and a path from $y$ to~$x$, for every pair of vertices $x$ and $y$. A subquiver \(Q' \subset Q\) is a \emph{strongly connected component} of \(Q\) if it is strongly connected and maximal with respect to this property. The \emph{condensation} \(c(Q)\) of a quiver~\(Q\) is the quiver with the strongly connected components of~\(Q\) as vertices; for two distinguished vertices~\(X\) and \(Y\) there is a directed edge \((X,Y)\) in \(c(Q)\) if and only if there is an edge \((x,y)\) in~\(Q\) for some \(x \in X\) and \(y \in Y\).  Observe that, by definition, the condensation of a quiver does not have multiple edges between two given vertices, and that the condensation of a directed cycle is the quiver with one vertex and a loop. When taken in the category of directed graphs where the morphisms are required to be edge preserving, condensation does not always yield a functor (see e.g.\ the text after~\cite[Remark~1.11]{persistentHH}). However,   condensation yields an endofunctor on  quivers. 

\section{The reachability category of a quiver}\label{sec:reach}
In this section, we introduce the  reachability category of a finite quiver. We provide a comparison with the path category and study related algebraic and topological properties. 

\subsection{The category \(\Reach_{Q}\)}
Let $Q$ be a finite quiver.

\begin{definition}\label{def:reach_G}
The  \emph{reachability category} \(\Reach_{Q}\)  is the category with objects the vertices of~\(Q\), and for $v,w\in Q$, the Hom-set \(\Reach_{Q}(v,w)\) is defined as follows:
\[
\Reach_{Q}(v,w)\coloneqq 
\begin{cases}
* & \text{if there is a path from $v$ to $w$ in $Q$}  \\
\emptyset & \text{ otherwise}\
\end{cases}
\]
The Hom-set \(\Reach_{Q}(v,v)\) is defined as the identity at \(v\).
\end{definition}

We use  the term \emph{reachability} in analogy with graph theory. In graph theory, in fact, the notion of reachability of a vertex \(w\) from a vertex \(v\) refers to the existence of a path from \(v\) to~\(w\). Definition~\ref{def:reach_G} is the direct categorical extension of this notion.

Recall that a category $\C$ is called \emph{thin} if for any pair of objects $c,c'\in \C$ there is at most one morphism $c\to c'$ between them. By definition, $\Reach_{Q}(v,w)$ has a single morphism if and only if there is a path in $Q$ from $v$ to $w$.

\begin{lemma}\label{lem:thin}
The reachability category of a quiver $Q$ is a thin category. 
\end{lemma}

Recall that an \(EI\)-category is a category in which every endomorphism is an isomorphism. The identity in \(\Reach_Q\) is the only endomorphism at each vertex, and is invertible, hence we get:

\begin{corollary}\label{cor:EI}
	The reachability category  is an \(EI\)-category.
\end{corollary}

The reachability category is strictly related to the path category. It is therefore expected that, in some cases, these categories are isomorphic. Recall that a  quiver is a \emph{tree} if its underlying undirected graph is a tree. Then, the following is straightforward from the definitions:

\begin{remark}\label{prop:polytree}
    If $Q$ is a tree, then the categories \(\Path_{Q}\) and \(\Reach_{Q}\) are isomorphic. 
\end{remark}

Remark~\ref{prop:polytree} does not give a complete characterisation of quivers $Q$ for which \(\Path_{Q}\) and \(\Reach_{Q}\) are isomorphic categories. For example, consider the quiver illustrated in Figure~\ref{fig:altquiv}.
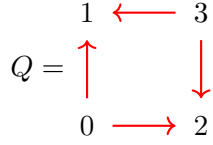
\begin{figure}
\begin{tikzpicture}
\tikzstyle{point}=[circle,thick,draw=black,fill=black,inner sep=0pt,minimum width=2pt,minimum height=2pt]
\tikzstyle{arc}=[shorten >= 3pt,shorten <= 3pt,->, thick, red]
\node[] (0) at  (0,0) {0};
\node[] (1) at  (0,1.5) {1};
\node[] (2) at  (1.5,0) {2};
\node[] (3) at  (1.5,1.5) {$3$};
	
\draw[arc] (0) to (1);
\draw[arc] (0) to (2);
\draw[arc] (3) to (1);
\draw[arc] (3) to (2);
	
\node at (-0.65,0.75) {\(Q=\)};
\end{tikzpicture}
\caption{An alternating quiver on four vertices.}
\label{fig:altquiv}
\end{figure}
Then, it is easy to see that \(\Path_{Q}\) and \(\Reach_{Q}\) are also isomorphic, even though $Q$ is not a tree. In order to completely characterise the family of finite quivers for which the path category and the reachability category are isomorphic, we  introduce the  notion of \emph{quasi-bigons}. Let~$B_{m,n}$ be the quiver illustrated in Figure~\ref{fig:nbigon}. We use the convention that, when $m,n=0$,  $B_{0,0}$ denotes the quiver on vertices~$x$ and $y$ with two edges from $x$ to $y$ and no other intermediate vertex.

\begin{figure}
\begin{tikzpicture}[baseline=(current bounding box.center)]
		\tikzstyle{point}=[circle,thick,draw=black,fill=black,inner sep=0pt,minimum width=2pt,minimum height=2pt]
		\tikzstyle{arc}=[shorten >= 8pt,shorten <= 8pt,->, thick]
		
		\node[left] (v0) at (0,0) {$x$};
  		\node (x) at (0,0) {};
		\draw[fill] (0,0)  circle (.05);
  
		\node[above] (v1) at (1.5,1) {$v_1$};
  		\node (vv1) at (1.5,1) {};
  
		\draw[fill] (1.5,1)  circle (.05);
		\node[below] (w1) at (1.5,-1) {$w_1$};
  		\node (ww1) at (1.5,-1) {};

		\draw[fill] (1.5,-1)  circle (.05);

		\node (v2) at (3,1) {\dots};
  		\node (w2) at (3,-1) {\dots};

		\node[above] (v4) at (4.5,1) {$v_m$};
  		\node (vv4) at (4.5,1) {};
		\draw[fill] (4.5,1)  circle (.05);
  		\node[below] (w4) at (4.5,-1) {$w_n$};
      		\node (ww4) at (4.5,-1) {};
		\draw[fill] (4.5,-1)  circle (.05);
  
		\node[right] (v5) at (6,0) {$y$};
  	\node (y) at (6,0) {};
		\draw[fill] (6,0)  circle (.05);
		
		\draw[thick, red, -latex] (x) to[bend left] (vv1);
  		\draw[thick, red, -latex] (x) to[bend right] (ww1);

\draw[thick, red, -latex] (vv1) to[] (v2);
  		\draw[thick, red, -latex] (ww1) to[] (w2);

\draw[thick, red, -latex] (v2) to[] (vv4);
  		\draw[thick, red, -latex] (w2) to[] (ww4);
\draw[thick, red, -latex] (ww4) to[bend right] (y);
  		\draw[thick, red, -latex] (vv4) to[bend left] (y);
\end{tikzpicture}
	\caption{The quiver $B_{m,n}$.}
	\label{fig:nbigon}
\end{figure}
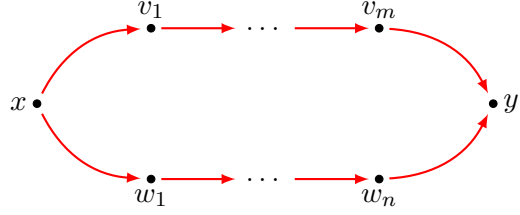

\begin{definition}
     We say that $B$ is a quasi-bigon of a quiver $Q$ if it is a subquiver of $Q$ isomorphic to $B_{m,n}$ for some $m,n\geq 0$. 
\end{definition}

If $B_{m,n}\cong Q$ {for some $m,n$}, we say that $Q$ itself is a quasi-bigon. In the same notations, if there is a path in \(Q\) from \(y\) to \(x\), we call it the \emph{diagonal} of \(B\). The following structural result will be used in Section \ref{sec:commuting_algebras}.
\begin{lemma}\label{lem:strong_qbigon}
    A quasi-bigon $B$ in a quiver \(Q\) is strongly connected if and only if it has a diagonal in \(Q\).
\end{lemma}
\begin{proof}
    Following the notation in \eqref{fig:nbigon}, first note that if there is a \(y - x\) path, then there is a path between any pair of vertices from the \(v_i\)'s and \(w_j\)'s. In the other direction, assume that there is no \(y - x\) path. Then, \(B\) is not strongly connected.
\end{proof}

\begin{proposition}\label{prop:char_diff_reach_path}
Let $Q$ be a finite connected quiver. Then, the categories \(\Path_{Q}\) and \(\Reach_{Q}\) are isomorphic if and only if $Q$ does not contain directed cycles nor quasi-bigons.
\end{proposition}

\begin{proof}
\(\Leftarrow \colon\) Take any two vertices \(v\) and \(w\) of $Q$. As there are no directed cycles, nor quasi-bigons in $Q$, it follows that there is at most one  path from \(v\) to \(w\). Then, the categories \(\Path_{Q}\) and \(\Reach_{Q}\) are isomorphic.

\(\Rightarrow \colon\) Assume that the two categories are isomorphic. They both have the vertices of $Q$ as set of objects, and there is a bijection between \(\Path_{Q}(v,w)\) and $\Reach_Q(v,w)$ for each pair of vertices $v,w$ of $Q$. Assume first $v=w$; then $\Reach_Q(v,v)$ contains only the identity of $v$. Therefore in \(\Path_{Q}(v,v)\) there is only a single morphism, implying the non-existence of directed cycles at $v$. Assume now $v\neq w$. Then, either $\Reach_Q(v,w)=\emptyset=\Path_Q(v,w)$, and there are no paths in $Q$ from $v$ to $w$, or $\Reach_Q(v,w)=\{\Gamma\}=\Path_Q(v,w)$ for a \(v-w\) path \(\Gamma\). In the latter there is precisely one path in $Q$ from $v$ to $w$. Hence, there are no quasi-bigons in $Q$.  
\end{proof}

Note that Proposition~\ref{prop:char_diff_reach_path} provides a categorical enhancement of Remark~\ref{rem:pathalgincalg}. Furthermore, the proof suggests an algorithmic way to verify whether the categories \(\Path_{Q}\) and \(\Reach_{Q}\) are isomorphic. This amounts to checking whether there are simple paths creating cycles or quasi-bigons. To make this description more rigorous, one can use the notion of contractions and path reductions. As such constructions might be of independent interest in handling quivers, we provide the complete description. 

Recall that the \emph{contraction} of a quiver~$Q$ with respect to the edge $e$ is the quiver $Q/e$ obtained from~$Q$ by contracting $e$ to a point. In other words, the edge~$e=(v,v')$ is removed, and its source~$v$ and its target~$v'$ are identified into a new vertex~$w$; edges incident to $v$ or $v'$, in $Q$, are set to be incident to $w$ in $Q/e$. We will only allow contractions of edges of type $(v,w)$ with $v\neq w$; this means that we do not allow contractions of loops. More generally, one can consider contractions of simple paths, see also \cite[Section~1.3]{bang2008digraphs}. We shall consider here contractions of maximal simple paths.

Let \(\Gamma=(e_0,e_1,\dots,e_n)\) be a \emph{simple} path of a quiver $Q$ that is maximal with respect to inclusion, i.e.\ there is no other directed simple path~$\Gamma'$ in $Q$ properly containing~$\Gamma$. 

\begin{definition}\label{def:pathcontr}
 For a finite quiver $Q$, the \emph{path contraction} of the simple path~$\Gamma=(e_0,e_1,\dots,e_n)$ in $Q$ is the quiver $Q/\Gamma$ obtained from $Q$ by contracting all the edges of~$\Gamma$, but $e_0$.
 \end{definition}

Note that, if $\Gamma$ is a single edge not contained in any longer simple path, then the path contraction of~$\Gamma$ in $Q$ yields an isomorphism.  

\begin{example}\label{ex:pathred}
   If $Q$ is the quiver $B_{m,n}$ and $\Gamma$ is the simple path $((x,v_1),(v_1,v_2),\dots,(v_m,y))$ represented in Figure~\ref{fig:nbigon}, then the path contraction of~$\Gamma$ in~$B_{m,n}$ is isomorphic to the quiver $B_{0,n}$; analogously, the path contraction of the path $((x,w_1),\dots, (w_n,y))$ in $B_{m,n}$ is isomorphic to the quiver $B_{m,0}$.
\end{example}

We want to iterate the procedure of path contraction described in Definition~\ref{def:pathcontr}. Consider an ordering on the set of maximal simple paths of~$Q$, and call such set~$\mathfrak{P}$.

 \begin{definition}
The \emph{path reduction} of a quiver $Q$, with respect to a given ordering on its maximal simple paths, is the quiver~$P(Q)$ obtained from $Q$ by path contraction of each element of~$\mathfrak{P}$.
\end{definition}

Note that contractions are not morphisms in the category~$\Quiver$. Nevertheless, composition of contractions is commutative up to isomorphism of quivers, and the procedure of iteratively contracting maximal simple paths yields again a quiver. 

\begin{example}
    Let $Q$ be a directed cycle with at least two edges. Then, its path reduction is isomorphic to the quiver with two vertices and two directed edges with the opposite orientation. Analogously, the path reduction of a quasi-bigon $B_{m,n}$ is the quiver $B_{0,0}$. Observe that, in general, directed cycles in a quiver generate either multiple edges  or loops; see also Figure~\ref{fig:nbigraph}.
\end{example}

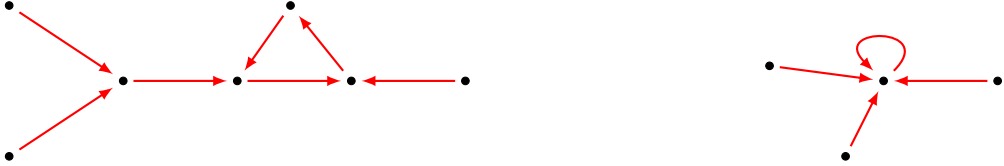
\begin{figure}[h]
\begin{tikzpicture}[baseline=(current bounding box.center)]
		\tikzstyle{point}=[circle,thick,draw=black,fill=black,inner sep=0pt,minimum width=2pt,minimum height=2pt]
		\tikzstyle{arc}=[shorten >= 8pt,shorten <= 8pt,->, thick]		
	
  		\node (x) at (0,1) {};
		\draw[fill] (0,1)  circle (.05);

        \node (z) at (0,-1) {};
		\draw[fill] (0,-1)  circle (.05);
  
		\node (v1) at (1.5,0) {};  
		\draw[fill] (1.5,0)  circle (.05);

        \node (v2) at (3,0) {};  
		\draw[fill] (3,0)  circle (.05);

        \node (v3) at (4.5,0) {};  
		\draw[fill] (4.5,0)  circle (.05);

        \node (v4) at (6,0) {};  
		\draw[fill] (6,0)  circle (.05);
  
		\node (w) at (3.7,1) {};  
		\draw[fill] (3.7,1)  circle (.05);

		\draw[thick, red, -latex] (x) to (v1);
  		\draw[thick, red, -latex] (z) to (v1);
        \draw[thick, red, -latex] (v1) to (v2);

        \draw[thick, red, -latex] (v2) to (v3);

        \draw[thick, red, -latex] (v4) to (v3);

        \draw[thick, red, -latex] (v3) to (w);

        \draw[thick, red, -latex] (w) to (v2);

        \node (xx) at (10,0.2) {};
		\draw[fill] (10,0.2)  circle (.05);

        \node (zz) at (11,-1) {};
		\draw[fill] (11,-1)  circle (.05);
  
		\node (vv1) at (11.5,0) {};  
		\draw[fill] (11.5,0)  circle (.05);

        \node (vv2) at (13,0) {};  
		\draw[fill] (13,0)  circle (.05);

        \draw[thick, red, -latex] (xx) to (vv1);
        \draw[thick, red, -latex] (zz) to (vv1);
        \draw[thick, red, -latex] (vv2) to (vv1);
        \draw[thick, red, -latex] (vv1) to[loop] (vv1);
        \end{tikzpicture}
	\caption{A quiver on the left, and its path reduction on the right.}
	\label{fig:nbigraph}
\end{figure}

The path reduction of a finite quiver $Q$ has no simple paths of length $\geq 2$, and no directed cycles of length $\geq 3$. An orientation ${o}$ on an unoriented graph $G$ is called \emph{alternating} if there exists a partition~$V\sqcup W$ of $V(G)$ such that all elements of $V$ have indegree $0$ and all elements of $W$ have outdegree~$0$ (cf.\ \cite[Definition~2.7]{jason}), see also Figure \ref{fig:altquiv}. We call a quiver~$Q$ alternating if its orientation is alternating. The existence of an alternating orientation is equivalent to $G$ being a bipartite graph. 

\begin{example}
    Let $Q$ be an alternating quiver. Then, $Q$ is isomorphic to its path reduction. In fact, in an alternating quiver there are no simple paths of length $2$.
\end{example}

The effect of taking the path reduction of a quiver is that it reduces the length of simple paths in $Q$. As we do not allow contractions of loops, and contractions do not create new cycles, they preserve the homotopy type of the quivers. A quiver is called \emph{simple} if it has no loops nor multiple edges. Let $Q$ be a finite connected quiver, equipped with an ordering of its maximal simple paths. Then, Proposition~\ref{prop:char_diff_reach_path}  directly implies the following alternative condition for path categories and reachability categories to be isomorphic:

\begin{corollary}
    The categories \(\Path_{Q}\) and \(\Reach_{Q}\) are isomorphic if and only if the path reduction $P(Q)$ of $Q$ is a simple alternating quiver. 
\end{corollary}

\begin{proof}
By Proposition~\ref{prop:char_diff_reach_path}, the categories \(\Path_{Q}\) and \(\Reach_{Q}\) are isomorphic if and only if $Q$  does not contain directed cycles nor quasi-bigons. Now observe that $Q$  does not contain directed cycles nor quasi-bigons if and only if, for any chosen ordering on its maximal simple paths, the  path reduction of $Q$ is an alternating directed graph, with no loops nor multiple edges.
\end{proof}

\subsection{Functoriality}
We show here that going from quivers to  reachability categories is functorial. This will be needed in applications in Section \ref{sec:applications}. Let $\mathbf{Thin}$ denote the category of small thin categories, which is the full subcategory of  $\Cat$ of thin categories. 

\begin{proposition}
    \label{prop:reach_functoriality}
Taking the reachability category yields a functor
\begin{equation*}\label{eq:reachfunct}
\mathrm{Reach}\colon \Quiver\to \mathbf{Thin} 
\end{equation*}
from the category of quivers to the category of thin categories.
\end{proposition}

\begin{proof}
By Lemma \ref{lem:paths_to_paths} a morphism $\phi\colon Q\to Q'$ of quivers sends a path in $Q$ to a path in~\(Q'\). By definition, $\phi$ is a natural transformation. Then, define $\mathrm{Reach}(\phi)\colon \Reach_Q\to\Reach_{Q'}$ to be the functor induced by the natural transformation $\phi$ on the objects, and such that, to a non-trivial morphism $f\colon v\to w$ in $\Reach_Q$, it associates the unique morphism between $\phi(v)$ and~$\phi(w)$. It is easy to see that $\mathrm{Reach}$ respects compositions of morphisms of quivers, and that it yields a functor from the category of quivers to the category of small categories. Now, by Lemma~\ref{lem:thin}, the reachability category is a thin category; hence, the statement follows.
\end{proof}

\begin{remark}\label{rem:reachascongr}
The reachability category \(\Reach_{Q}\) can alternatively be  constructed by taking quotients of the Hom-sets of the path category $\Path_Q$. In fact, one can define a functor $F\colon \Path_{Q} \ra \Path_{Q} /\!\! \sim  $ that is the identity on objects, and such that  the congruence~\(\sim\) collapses all paths between two objects to a single morphism. However, this construction would lead us to define functorial quotients on \(\mathbf{Cat}\). Proposition~\ref{prop:reach_functoriality} provides a direct proof of the functoriality.
\end{remark}
    
\subsection{Topological properties}\label{sec:topological_properties}
We proceed with a  comparison of path categories and reachability categories from a topological viewpoint. To do so, we study the nerves of these categories. In the follow-up, by geometric realisation $|Q|$ of a quiver $Q$, we shall mean the geometric realisation of $Q$ as an undirected graph (i.e.~by forgetting the directions of the edges and then taking the realisation of the obtained CW-complex). Analogously, for a simplicial set~$S$ (e.g.~the nerve of a category), $|S|$ denotes its geometric realisation. First, recall the following result:

\begin{proposition}[{\cite[Ex.~4.3]{citterio}}]\label{prop:hogr=homcat}
	The classifying space $|\Nerve(\mathbf{Path}_Q)|$ of the nerve of the path category  of a quiver $Q$ has the homotopy type of the geometric realisation $|Q|$.
\end{proposition}

By Proposition~\ref{prop:hogr=homcat}, the classifying space of the nerve \(\Nerve(\Path_Q)\) has the homotopy type of a wedge of circles. We provide some examples.

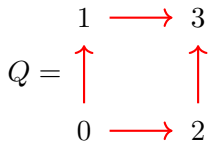
\begin{figure}
\begin{tikzpicture}
\tikzstyle{point}=[circle,thick,draw=black,fill=black,inner sep=0pt,minimum width=2pt,minimum height=2pt]
\tikzstyle{arc}=[shorten >= 3pt,shorten <= 3pt,->, thick, red]
\node[] (0) at  (0,0) {0};
\node[] (1) at  (0,1.5) {1};
\node[] (2) at  (1.5,0) {2};
\node[] (3) at  (1.5,1.5) {$3$};
	
\draw[arc] (0) to (1);
\draw[arc] (0) to (2);
\draw[arc] (1) to (3);
\draw[arc] (2) to (3);
	
\node at (-0.65,0.75) {\(Q=\)};
\end{tikzpicture}
\caption{The quiver $Q$.}
\label{fig:squaredec}
\end{figure}

\begin{example}
Consider the quiver \(Q\) as illustrated in Figure~\ref{fig:squaredec}. Observe that, in virtue of Proposition~\ref{prop:char_diff_reach_path}, the associated path category and reachability category are not isomorphic. By Proposition~\ref{prop:hogr=homcat}, the nerve of \(\Path_Q\) is homotopic to the circle~\(S^1\). The nerve of \(\Reach_Q\), on the other hand, is contractible since \(\Reach_Q\) has an initial object 0. The graph representations of these two categories are illustrated in Figure~\ref{fig:pathreachGG},
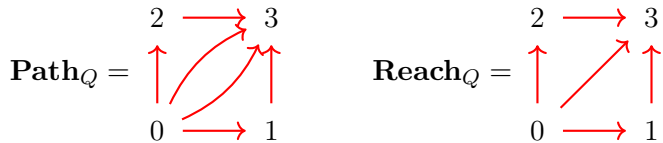
\begin{figure}
\begin{tikzpicture}
\tikzstyle{point}=[circle,thick,draw=black,fill=black,inner sep=0pt,minimum width=2pt,minimum height=2pt]
\tikzstyle{arc}=[shorten >= 3pt,shorten <= 3pt,->, thick,red]

\node[] (0') at (5,0) {0};
\node[] (1') at (6.5,0) {1};
\node[] (2') at (5,1.5) {2};	
\node[] (3') at (6.5,1.5) {3};

\draw[arc] (0') to (1');
\draw[arc] (0') to (2');
\draw[arc,bend left = 20] (0') to (3');
\draw[arc,bend right = 20] (0') to (3');
\draw[arc] (1') to (3');
\draw[arc] (2') to (3');

\node at (3.85,0.75) {\(\Path_Q=\)};

\node[] (0'') at (10,0) {0};
\node[] (1'') at (11.5,0) {1};
\node[] (2'') at (10,1.5) {2};	
\node[] (3'') at (11.5,1.5) {3};

\draw[arc] (0'') to (1'');
\draw[arc] (0'') to (2'');
\draw[arc] (0'') to (3'');
\draw[arc] (1'') to (3'');
\draw[arc] (2'') to (3'');

\node at (8.75,0.75) {\(\Reach_Q=\)};	
\end{tikzpicture}
\caption{The path category and the reachability category of the quiver $Q$ in Figure \ref{fig:squaredec}.}
\label{fig:pathreachGG}
\end{figure}
where we have omitted the identity morphisms on the vertices.
\end{example}

Recall that an equivalence of categories induces a homotopy equivalence between the nerves. 

\begin{example}
    Consider the quiver $L_n$ illustrated in Figure~\ref{fig:nbidlin}. Its path category is homotopic to $\bigvee_{i=1}^n S^1$. However, the associated reachability category is equivalent to the category with one object and one identity morphism. In fact, imposing the reachability condition implies that the morphisms $v_i\to v_{i+1}$ and $v_{i+1}\to v_i$ are inverses of one another. As an equivalence of categories preserves the homotopy type of the nerve, we get $|\Nerve(\Reach_{L_n})|\simeq *$. 
\end{example}

\begin{figure}[h]
\begin{tikzpicture}[baseline=(current bounding box.center)]
		\tikzstyle{point}=[circle,thick,draw=black,fill=black,inner sep=0pt,minimum width=2pt,minimum height=2pt]
		\tikzstyle{arc}=[shorten >= 8pt,shorten <= 8pt,->, thick, red]
		
		\node[left] (v0) at (0,0) {$v_0$};
		\draw[fill] (0,0)  circle (.05);
		\node (v1) at (1.5,0) {};
		\draw[fill] (1.5,0)  circle (.05);
		\node (v2) at (3,0) {\dots};
		\node (v4) at (4.5,0) {};
		\draw[fill] (4.5,0)  circle (.05);
		\node[right] (v5) at (6,0) {$v_n$};
		\draw[fill] (6,0)  circle (.05);
		
		\draw[thick, red, -latex] (v0) to[bend left] (v1);
  		\draw[thick, red, -latex] (v1) to[bend left] (v0);
        \draw[thick, red, -latex] (v1) to[bend left] (v2);
  		\draw[thick, red, -latex] (v2) to[bend left] (v1);

        \draw[thick, red, -latex] (v2) to[bend left] (v4);
  		\draw[thick, red, -latex] (v4) to[bend left] (v2);
        \draw[thick, red, -latex] (v5) to[bend left] (v4);
  		\draw[thick, red, -latex] (v4) to[bend left] (v5);
        \end{tikzpicture}
	\caption{The quiver $L_n$.}
	\label{fig:nbidlin}
\end{figure}
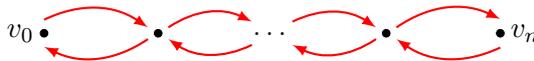

Motivated by the previous examples, we can state the following: 

\begin{proposition}
    \label{lem:strongconnecomp}
	Let \(Q\) be a strongly connected quiver. Then, \(\Reach_{Q}\) is contractible.
\end{proposition}

\begin{proof}
    Choose an object $q$ of $Q$, and let $\mathbf{1}$ be the category with the one object \(q\) and a single identity morphism. Then, the functor $F\colon \Reach_{Q} \to \mathbf{1}$ is an equivalence of categories. Therefore, the nerve of $\Reach_{Q}$ is homotopic to the nerve of $\mathbf{1}$, which is contractible. 
\end{proof}

By Proposition~\ref{prop:char_diff_reach_path} and Proposition~\ref{prop:hogr=homcat}, the nerve~$|\Nerve(\Reach_Q)|$ of  a finite connected quiver~$Q$ with no directed cycles nor quasi-bigons is  homotopic to the geometric realisation $|Q|$. On the other hand, if $Q$ contains directed cycles or quasi-bigons, and $H$ is a strongly connected component of $Q$, then by Proposition~\ref{lem:strongconnecomp} $\Reach_H$ is contractible. {Taking}  reachability categories is functorial. Therefore,  the inclusion of $H$ in $Q$ induces a functor $\Reach_H\to \Reach_Q$, hence a continuous map  $\ast\simeq |\Nerve(\Reach_H)| \to |\Nerve(\Reach_Q)|$ which is a cofibration. In other words, replacing a strongly connected component in $Q$ with a single vertex does not change the homotopy type of $|\Nerve(\Reach_Q)|$ (see also \cite[Lemma~10.2]{BjorTopMeth}). Although, by Proposition~\ref{prop:char_diff_reach_path}, the categories $\Path_{Q}$ and $\Reach_Q$ are not isomorphic when $Q$ contains directed cycles, we get the following:

\begin{proposition}\label{prop:cond}
    Let $Q$ be a finite connected quiver with no quasi-bigons. Then, there is a homotopy equivalence 
    \[
    |\Nerve(\Path_{c(Q)})| \simeq |\Nerve(\Reach_Q)| \ ,
    \]
    where $c(Q)$ is the condensation  of $Q$. 
    \end{proposition}
\begin{proof}
In view of Proposition~\ref{prop:char_diff_reach_path}, $|\Nerve(\Path_{c(Q)})| $ and $ |\Nerve(\Reach_{c(Q)})|$ are homotopy equivalent. Proceeding as in the proof of Proposition~\ref{lem:strongconnecomp}, collapsing the directed cycles of $Q$ does not change the homotopy type of $\Reach_Q$, yielding an homotopy equivalence between $|\Nerve(\Reach_Q)|$ and $ |\Nerve(\Reach_{c(Q)})|$. The statement follows.  
\end{proof}

Proposition~\ref{prop:cond} does not hold if the quiver $Q$ is the Hasse diagram of a poset which is not a tree (as a graph); hence, we can not exhibit a complete classification of the  homotopy classes of $\Reach_Q$. In fact, we have the following example:

\begin{example}\label{ex:face_poset_and_Reach}
Let $P$ be the face poset of a simplicial complex~$X$. Its associated reachability category is the poset $P$ itself, seen as a category. The nerve of $\Reach_{P}$ is homotopy equivalent to the barycentric subdivision of $X$. Therefore, if the simplicial complex~$X$ has non-trivial homotopy groups in degree $\geq 2$, and $Q$ is the Hasse diagram of~$P$, then the nerve of $\Reach_Q$ has the same (non-trivial) homotopy type of $X$.  On the other hand, the homotopy groups of the nerve of a path category are always trivial in degree~$\geq 2$ due to Proposition~\ref{prop:hogr=homcat}.
\end{example}
This example suggests that we can find interesting (topological and categorical) information in $\Reach_Q$, provided $Q$ has quasi-bigons. In fact, the homology of $\Reach_Q$ has recently found important applications in  homology theories of digraphs; in particular, the homology of $\Reach_Q$ can be interpreted as the limit of  the magnitude-path spectral sequence~\cite{asao,hepworth2023reachability}.

\section{The reachability poset}
The  aim of this section is to associate to each quiver a poset: the reachability poset. As homological invariants of reachability categories will also be invariants of reachability posets, we will use this equivalence in the applications of Section~\ref{sec:applications}. 

A \emph{preordered set}, or simply \emph{preorder}, is a pair $(X,\leq )$ consisting of a set $X$ and a binary relation $\leq $ that is reflexive and transitive. Therefore, any preordered set yields a thin category by declaring a unique morphism \(x \ra y\) whenever \(x \leq y\). Vice versa, a thin category yields, up to isomorphism of categories, a preordered set by setting the relation $x\leq y$ between objects $x$ and $y$ if and only if there is the unique morphism $x\to y$. We can identify the category of thin categories with the category of preorders. Given a preorder, there is a standard way to construct a poset: first, define an equivalence relation $\simeq$ on the elements of the preorder such that $p\simeq q$ if and only if $p\leq q$ and $q\leq p$. Then, define the poset as the quotient obtained from the preorder via the given equivalence relation. As we are interested in extending this process to a functor from thin categories to posets, we make use of the standard notion of posetal reflection. For the sake of completeness, we present the construction in details. 

Recall that a category \(\C\) is \emph{skeletal} if each of its isomorphism classes has just one object. The \emph{skeleton} \(\sk \C\) of  \(\C\) is the unique (up to isomorphism) skeletal category equivalent to \(\C\). Assuming the axiom of choice, every category has a skeleton obtained by choosing one object in each isomorphism class of $\C$, and then by defining \(\sk\C\) to be the full subcategory on this collection of objects (cf.\ \cite[Proposition~2.6.4]{richter}). The skeleton construction yields  an equivalence of categories \(\sk\C \hookrightarrow \C\) between the skeletal subcategory and the category $\C$ itself. However, this construction cannot be promoted to an endofunctor $ \sk \colon \Cat \to \Cat $; it yields in fact a pseudofunctor. Roughly speaking, for categories $\bA,\bB$ and a functor $F\colon \bA\to \bB$, it is possible to \textit{choose} a functor \(\sk F \colon \sk\bA \ra \sk\bB\), but these choices will not be strictly functorial. To formalise it, we first need to recall the definition of reflection, see e.g.\ \cite[Definition~4.16]{concrcat}:

\begin{definition}
Let $\C$ be a subcategory of $\D$, and $d$ an object of $\D$. A \emph{reflection} for $d$ is a morphism $\rho\colon d\to c$ in $\D$ from $d$ to $c\in \C$ such that the following universal property is satisfied: for any $f\colon d\to c'$ in $\D$ with $c'\in \C$, there exists a unique morphism $f'\colon c\to c'$ of $\C$ such that this diagram 
\begin{center}
 \begin{tikzcd}
d\arrow[dr,"f"'] \arrow[r, "\rho"] & c \arrow[d, "f'"]\\
  & c'
 \end{tikzcd}    
 \end{center}
 commutes.
\end{definition}

A subcategory $\C$ of $\D$ with the property that each object $d\in \D$ has a reflection is called a \emph{reflective} subcategory. Equivalently, a full subcategory $\C$ of a category $\D$ is said to be {reflective} in $\D$ if the inclusion functor from $\C$ to $\D$ has a left adjoint. For reflective subcategories, the following result is standard:

\begin{proposition}[{\cite[Proposition 4.22]{concrcat}}]
    Let $\C$ be a reflective subcategory of $\D$, and for each $d\in \D$ let $\rho_d\colon d\to c_d$ be a reflection. Then, there exists a unique functor $R\colon \D\to \C$ such that:
    \begin{itemize}
        \item $R(d)=c_d$ for all $d$ in $\D$;
        \item for each morphism $f\colon d\to d'$ in $\D$, the diagram 
    \begin{center}
 \begin{tikzcd}
d\arrow[d,"f"'] \arrow[r, "\rho_d"] & R(d) \arrow[d, "R(f)"]\\
 d' \arrow[r, "\rho_{d'}"] & R(d')
 \end{tikzcd}    
 \end{center}
 commutes.
    \end{itemize}
\end{proposition}

The full subcategory of skeletal categories is reflective in the category of small categories (see e.g.\ \cite[Corollary 4.2]{BORCEUX2023109110}, along with \cite{MR4130917}). Analogously, this holds true for the category of thin categories. As a consequence, after identifying thin categories and preorders, we can construct a functor from the subcategory of preorders in \(\Cat\) to the category of posets as follows. Let $\Preord$ be the category of preorders and order-preserving maps, and let $\Poset$ be the full subcategory  of posets.
Taking the  quotient  turning a preorder into a poset is a reflection (cf.\ \cite[Section~4.17]{concrcat}). Then, there is a unique functor $L\colon \Preord\to \Poset$ induced by the described reflections. Sometimes, this functor is called a \emph{posetal reflection}, see e.g.\ \cite{raptis, zbMATH07066042, puca2023obstructions}. Furthermore, each other choice would yield a different  functor,  but all such functors are naturally isomorphic. Therefore, using the described choice of functors, we get the composition
\begin{equation}\label{eq:comporeach}%
\cR\colon
\Quiver
\xrightarrow{\mathrm{Reach}}\mathbf{Thin}\cong\Preord
\xrightarrow{L}\Poset
\end{equation}
that associates to a quiver the poset $\cR(Q)$.

\begin{definition}
    For a finite quiver $Q$, the poset~$\cR(Q)$ is called the \emph{reachability poset} of $Q$.
\end{definition}

By construction,  the objects of $\cR(Q)$ are the vertices of $Q$ modulo the equivalence relation~$\simeq $ which identifies $v$ and $w$ if and only if there are paths from $v$ to $w$ and from~$w$ to~$v$. Denote by~$[v]$ the equivalence class of $v$. Then,  $[v]\leq [w]$ if and only if there are representatives~$v$ of $[v]$ and $w$ of $[w]$, together with a directed path from $v$ to $w$ in $Q$.

For a poset $P$ and a quiver $Q$, we can look at the posets $P$ and $\cR(Q)$ as categories, and consider functors between them. With this convention, the reachability category and reachability poset satisfy the following universal property:

\begin{proposition}\label{prop:condensuni}
Let \(F\colon \Reach_Q \ra P\) be a functor. Then, there exists a unique functor \(G\colon \cR(Q)\to P\) making the diagram 
\begin{center}
 \begin{tikzcd}
Q \arrow[r,mapsto] & \Reach_Q \arrow[dr,"F"'] \arrow[r, "L"] & \cR(Q) \arrow[d,dotted, " G"]\\
  & & P
 \end{tikzcd}    
 \end{center}
 commute. 
\end{proposition}
\begin{proof}
    This follows directly from the properties of the reflection $L$.
\end{proof}

Consider the  forgetful functor $U\colon \Cat\to\Quiver$. Then, composition with the functor~$\cR$ of Equation~\eqref{eq:comporeach} yields an endofunctor
\begin{equation}\label{eq:reachgraph}
T\coloneqq U\circ \cR\colon\Quiver\to \Quiver
\end{equation}
of the category of (finite) quivers. 

\begin{remark}
    By construction, the quiver $T(Q)$ does not contain non-trivial directed cycles or multiple edges. If one forgets also the loops, the quiver~$T(Q)$ becomes acyclic.
\end{remark}

\begin{proposition}\label{prop:fcond}
    Let $Q$ be a finite quiver. Then $T(Q)$ is isomorphic to the condensation of the transitive closure of $Q$.
\end{proposition}

\begin{proof}
    The quiver $T(Q)$ is, by construction, obtained by first taking category $\Reach_Q$ of $Q$; this  can be identified with the transitive closure of $Q$. Now, each strongly connected component of~$Q$  yields a strongly connected component in the transitive closure. Such a component~$H$ is represented by a single vertex~$[h]$ in $T(Q)$. All edges connecting two strongly connected components $H$ and $H'$, say directed from $h$ in $H$ to $h'$ in $H'$, are sent to the edge $([h],[h'])$ in $T(Q)$. Hence, the vertices of $T(Q)$ are the strongly connected components of~$Q$, and there is an edge $([h],[h'])$ in $T(Q)$ if and only if there is an edge $(h,h')$ in the transitive closure of $Q$, with $h$ and $h'$ not strongly connected. This is enough to show that $T(Q)$ is isomorphic to the condensation  of the transitive closure of $Q$.
\end{proof}

As a consequence, if the quiver $Q$ is isomorphic to its transitive closure (e.g.\ when it is alternating), then $T(Q)$ is precisely the condensation of $Q$. 

\begin{example}
    Even though the condensation \(c(Q)\) and taking the underlying quiver \(T(Q)\) of the reachability poset \(\cR(Q)\) are very similar operations, they do not give the same result; the former is a quiver operation, while the latter factors through categories requiring the transitive closure in Proposition \ref{prop:fcond}. This is illustrated by the two examples below (omitting self-loops).
    
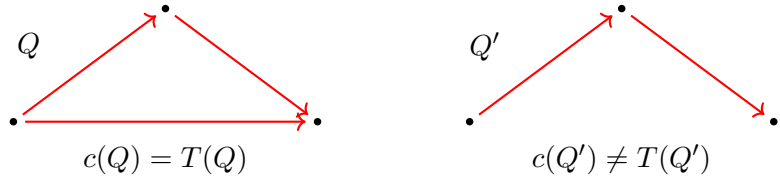
\begin{center}
    \begin{tikzpicture}
    \tikzstyle{point}=[circle,thick,draw=black,fill=black,inner sep=0pt,minimum width=2pt,minimum height=2pt]
    \tikzstyle{arc}=[shorten >= 3pt,shorten <= 3pt,->, thick, red]
    \node[point] (0) at  (0,0) {};
    \node[point] (1) at  (2,1.5) {};
    \node[point] (2) at  (4,0) {};
	
    \draw[arc] (0) to (1);
    \draw[arc] (0) to (2);
    \draw[arc] (1) to (2);
	
    \node at (0.2,1) {\(Q\)};

    \node[point] (0') at  (6,0) {};
    \node[point] (1') at  (8,1.5) {};
    \node[point] (2') at  (10,0) {};
	
    \draw[arc] (0') to (1');
    \draw[arc] (1') to (2');
	
    \node at (6.2,1) {\(Q'\)};

    \node at (2,-0.5) {\(c(Q) = T(Q)\)};
    \node at (8,-0.5) {\(c(Q') \ne T(Q')\)};
    \end{tikzpicture}
    \end{center}
\end{example}
    
\section{Applications}\label{sec:applications}

We show some applications of reachability categories to two main fields: finite-dimensional algebras and persistent homology. In particular, we show that commuting algebras are Morita equivalent to incidence algebras of reachability posets, and we conclude with the construction of a persistent Hochschild homology pipeline for quivers.

\subsection{Commuting algebras}\label{sec:commuting_algebras}
For a finite quiver $Q$ and $\mathbb{K}$ a field, let $\mathcal{C}(Q)\coloneqq \mathbb{K}Q/C$ be the \emph{commuting algebra} of $Q$ introduced in~\cite{green2023commuting}, i.e.~the path algebra $\mathbb{K}Q$ of $Q$ modulo its parallel ideal \(C\). We have the following characterisation of category algebras of reachability categories:

\begin{lemma}\label{lem:commutingalgiso}
The category algebra of $\Reach_Q$ is isomorphic to the commuting algebra $\mathcal{C}(Q)$.
\end{lemma}

\begin{proof}
By Example~\ref{ex:path algebras}, the category algebra of $\Path_Q$ is the  path algebra $\mathbb{K}Q$ of $Q$. The commuting algebra $\mathcal{C}(Q)$ is obtained from~$\mathbb{K}Q$ by taking the quotient with respect to the parallel ideal generated by all differences of finite directed paths in $Q$ with same source and target -- cf.~\cite{green2023commuting}. Consider the linear map 
\[
\mathbb{K}\Path_Q\longrightarrow \mathbb{K}\Reach_Q
\]
of vector spaces, induced by the functor \(F\colon \Path_{Q} \ra \Reach_Q=\Path_{Q} /\!\! \sim\) of Remark~\ref{rem:reachascongr}. The kernel of this linear map is precisely the vector subspace generated by differences of paths in $Q$ with same source and target, i.e.\ the parallel ideal of $\mathbb{K}Q$. The composition of paths is preserved, hence we have an algebra isomorphism between the category algebra of the reachability category and the commuting algebra.
\end{proof}

Recall that two unital rings are said to be Morita equivalent if and only if their categories of left (or right) modules are equivalent, cf.\ \cite[Chapter~6]{Smith1975FWA}. Then, we  recover  one of the  results of \cite{green2023commuting} in categorical setting:

\begin{theorem}\label{thm:enhance}
    Let $Q$ be a finite quiver and $\mathbb{K}$ a field. Then, the commuting algebra $\mathcal{C}(Q)$ is Morita equivalent to the incidence algebra of $\cR(Q)$.
\end{theorem}

\begin{proof}
    The category $\Reach_Q$ is equivalent to the poset $\cR(Q)$, seen as a category. As $Q$ is finite,  the associated category algebras are Morita equivalent by \cite[Proposition~2.2.4]{Xu_reps_of_cats}. By Lemma~\ref{lem:commutingalgiso} the category algebra of $\Reach_Q$ is isomorphic to the commuting algebra of $Q$, these are also Morita equivalent. Hence, the commuting algebra of $Q$ is Morita equivalent to the category algebra of the reachability poset~$\cR(Q)$, which is an incidence algebra. 
\end{proof}

If we restrict to the category $\Quiver_0$ of finite acyclic quivers, we get a functor
\begin{equation}
    \label{eq:quivalg}
    \mathcal{I}\coloneqq I\circ \cR\colon\Quiver_0 \to \mathbb{K}\text{-}\mathbf{Alg}
\end{equation}
which associates to a quiver $Q$ the incidence algebra of the reachability poset $\cR(Q)$. Then, Theorem~\ref{thm:enhance} says that the diagram
 \begin{center}
 \begin{tikzcd}
\Quiver_0\arrow[dr,"\mathcal{C}"'] \arrow[r, "\cR"] & \Poset \arrow[d, "I"]\\
  & \mathbb{K}\text{-}\mathbf{Alg}
 \end{tikzcd}    
 \end{center}
 is commutative up to Morita equivalence.
 
By \cite[Theorem~1]{stanley}, if the incidence algebras of two locally finite posets $P$ and $Q$ are isomorphic, as $\mathbb{K}$-algebras, then also $P$ and $Q$ are isomorphic, as posets. Note that the extension to the whole category of quivers does not yield the same result as in \cite[Theorem~1]{stanley}. However, we can infer the following: 

\begin{corollary}\label{cor:incidence_algebras_iso}
Let $\mathbb{K}$ be a field. If the commuting algebras of the  finite posets $P$ and $Q$ are isomorphic, as $\mathbb{K}$-algebras, then the reachability categories $\Reach_P$ and $\Reach_Q$ are isomorphic.   
\end{corollary}

 \begin{proof}
 Let $P$ and $Q$ be finite posets. Then, seen as as quivers, they generate the reachability categories $\Reach_P$ and $\Reach_Q$, which are still posets. In fact, they agree with the reachability posets~$\cR(P)$ and $\cR(Q)$. By assumption, the associated category algebras are isomorphic, and such isomorphism is reflected in an isomorphism between the incidence algebras of $\Reach_P$ and $\Reach_Q$. By  \cite[Theorem~1]{stanley}, the posets, hence the reachability categories, are also isomorphic.
 \end{proof}

We  provide a complete characterisation of quivers with Morita equivalent commuting algebras.

\begin{theorem}\label{thm:stley}
    Let $Q, Q'$ be finite quivers. Then, the  commuting algebras $\mathcal{C}(Q)$ and $\mathcal{C}(Q')$ are Morita equivalent if and only if the reachability posets $\cR(Q)$ and $\cR(Q')$ are isomorphic.
\end{theorem}

\begin{proof}
    By Lemma~\ref{lem:commutingalgiso}, the commuting algebras of the quivers $Q$ and $Q'$ are obtained by taking the category algebras of the reachability categories $\Reach_Q$ and $\Reach_{Q'}$. Thererefore, $\mathcal{C}(Q)$ and $\mathcal{C}(Q')$ are Morita equivalent if and only if the category algebras $\mathbb{K}\Reach_Q$ and $\mathbb{K}\Reach_{Q'}$ are Morita equivalent. Category algebras of acyclic categories are Morita equivalent if and only if they  are isomorphic, hence  the category algebras $\mathbb{K}\Reach_Q\cong \mathbb{K}\cR(Q)$ and $\mathbb{K}\Reach_{Q'}\cong \mathbb{K}\cR(Q')$ are Morita equivalent if and only if the incidence algebras $\mathbb{K}\cR(Q)$ and $\mathbb{K}\cR(Q')$ are isomorphic. By Stanley's theorem~\cite[Theorem~1]{stanley}, this  happens if and only if the reachability posets $\cR(Q)$ and $\cR(Q')$ are isomorphic.
\end{proof}
   
Recall that the global dimension of a ring $R$ is the  supremum of the set of projective dimensions of all $R$-modules. For a finite quiver $Q$, we denote by $\mathrm{diam}(Q)$ the maximal length across all directed simple paths in $T(Q)$. 

\begin{corollary}\label{cor:gldim}
    Let $Q$ be a finite quiver. Then,
    \[
    \mathrm{gl.dim} \, \mathcal{C}(Q) \leq \mathrm{diam}(Q).
    \]
\end{corollary}

\begin{proof}
The reachability category  of a finite quiver is an \(EI\)-category by Corollary~\ref{cor:EI}. Moreover, for each object $x$ of~$\Reach_Q$, the automorphism group of $x$ is trivial. Then, by \cite[Theorem~5.3.1]{Xu_thesis}, we have 
    \[
    \mathrm{gl.dim} \, \mathbb{K}\Reach_Q\leq \ell(\Reach_Q)
    \]
    where $\ell(\Reach_Q)$ is the maximal length of chains of non-isomorphisms in the poset~$\cR(Q)$. Note that each such chain is in bijection with a directed path in $T(Q)$, hence $\ell(\Reach_Q)=\mathrm{diam}(Q)$. The statement now follows from Lemma~\ref{lem:commutingalgiso}.
\end{proof}

Let \(\mathbf{B}_{1,1}\) be the poset generated by the quasi-bigon \(B_{1,1}\), i.e.~the directed square on four vertices. 
\begin{proposition}\label{prop:gl_dim_no_B11}
    Let \(Q\) be a finite quiver with at least one edge. Then
    \[
    \mathrm{gl.dim} \, \mathbb{K}\Reach_Q = \mathrm{gl.dim} \, \mathbb{K}\cR(Q) = 1
    \]
    if and only if the reachability poset \(\cR(Q)\) does not contain \(\mathbf{B}_{1,1}\) as a subposet.
\end{proposition}
\begin{proof}
    By  Morita equivalence, \(\mathrm{gl.dim} \, \mathbb{K}\Reach_Q = \mathrm{gl.dim} \, \mathbb{K}\cR(Q)\). Then, the statement follows from \cite[Theorem 4.2]{Mitchell_ordered_sets}.
\end{proof}
Note that if all the quasi-bigons in \(Q\) have diagonals, then by Lemma \ref{lem:strong_qbigon} \(\cR(Q)\) can not contain~\(\mathbf{B}_{1,1}\) as a subposet. Vice versa, if there is a quasi-bigon \(B_{m,n}\) with \(m,n \geq 1\) such that \(x\), \(y\) and a pair of vertices \(v_i\) and \(w_j\) are all in different strongly connected components, then the global dimension of the commuting algebra \(\mathbb{K}\Reach_Q\) is larger than 1. Finally, note that the condition in Proposition \ref{prop:gl_dim_no_B11} does not require \(Q\) to be a tree; for example, the alternating square in Figure \ref{fig:altquiv} does not have \(\mathbf{B}_{1,1}\) as a subposet.

\subsection{Persistent homology of quivers}\label{subsec:homological_stuff}
(Co)homology theories of directed graphs and quivers have become important tools in mathematics and science in general, largely due to the ubiquity of graph data in various research domains, and have gathered momentum in topological data analysis \cite{Reimann_2017,kaul,persistentHH,Riihimaki_simplicial_connectivities}. In topological data analysis one generally associates to a (undirected, directed) graph various types of homology groups; we refer to~\cite{persistentHH} for an excursus of possible homology theories. Among the various options, one can use the categorical framework, which allows to generally attain more refined invariants. Among others, one can associate homology groups to a quiver using the nerve construction. One can also study the homology of the category algebra $R\bC$ for (any) category \(\bC\) associated to a quiver~\(Q\); examples of this sort include Hochschild homology \(\mathrm{HH}\), or cyclic homology \cite{loday}, which are invariants associated to category algebras. 

For a quiver, a natural category to study is its path category. As the category \(\mathbf{Path}_Q\) is $1$-dimensional, its homological invariants  vanish beyond degree~1 (recall Proposition~\ref{prop:hogr=homcat}), and they are readily computable when the quiver, and the associated category, are acyclic. For Hochschild (co)homology in degrees $0$ and~$1$, we can resort to a well known result due to Happel (see also \cite[Proposition~4.4]{redondo}):

\begin{theorem}[\cite{happel}]\label{thmhapp}
	If $Q=(V,E,s,t)$ is a connected quiver without oriented cycles and~$\mathbb{K}$ is an algebraically closed field, then 
	\[
	\dim_{\mathbb{K}} \mathrm{HH}^i(A)=\dim_{\mathbb{K}} \mathrm{HH}_i(A)=
	\begin{cases}
		1 \quad &\text{ if } i=0 \\
		0 \quad &\text{ if } i>1 \\
		1-n +\sum_{e\in E}\dim_{\mathbb{K}} e_{t(e)}A e_{s(e)}  \quad &\text{ if } i=1
	\end{cases}
	\] 
	where $A=\mathbb{K}\Path_Q$ is the path algebra of $Q$,  $n=|V|$ is the number of vertices of $Q$ and $e_{t(e)}A e_{s(e)}$ is the subspace of $A$ generated by all the possible paths from $s(e)$ to $t(e)$ in $Q$.
\end{theorem}  

Recall that persistent homology is a functor \((\mathbb{R}, \leq) \ra \mathbf{FinVect}\) with values in finite dimensional vector spaces over a field \(\mathbb{K}\). In the topological setting this can be realised by taking the homology of a filtered simplicial complex. In \cite{persistentHH}, the following persistent Hochschild homology pipeline was introduced, where \(\mathrm{HH}\) is computed  on the category algebra of \(\Path(c(Q))\):
\begin{equation}\label{eq:compvar}
    (\mathbb{R}, \leq) \ra \mathbf{Digraph}\xrightarrow{c}\mathbf{Digraph}\xrightarrow{\mathbb{K}-} \mathbb{K}\text{-}\mathbf{Alg}\xrightarrow{\mathrm{HH}}\mathbf{FinVect} \ ;
\end{equation}
note that in this pipeline the condensation \(c\) facilitates the use of Happel's formula. On the other hand, condensation also renders the composition \eqref{eq:compvar} non-functorial. Then, we might wish to retrieve the functoriality using a different category, e.g.~\(\Reach_Q\), instead of the free category~\(\Path_Q\). In fact, recall that equivalent categories with finitely many objects have Morita equivalent category algebras. Then by \cite[Theorem 1.2.7]{loday} we get $\mathrm{HH}_*(\mathbb{K}\Reach_Q)\cong \mathrm{HH}_*(\mathbb{K}\cR(Q))$, and \(\cR(Q)\) (as an acyclic quiver) satisfies the assumptions of Theorem~\ref{thmhapp}. However, the following result sets up an obstacle to the definition of a functorial pipeline involving Hochschild homology of reachability categories.
\begin{proposition}[\cite{Xu_reps_of_cats}, Proposition 2.2.3]\label{prop:inj_functor_alg_homomorphism}
    A functor \(\mu \colon \D \ra \C\) between small categories extends linearly to an algebra homomorphism \(\overline{\mu} \colon \mathbb{K}\D \ra \mathbb{K}\C\) if and only if \(\mu\) is injective on objects of \(\D\).
\end{proposition}

Because the functor \(\Reach_Q \ra \cR(Q)\) collapses isomorphism classes to single objects there is no obvious functorial morphism of algebras $\mathbb{K}\Reach_Q \ra \mathbb{K}\cR(Q)$. Hence, the filtration step
\begin{center}
\begin{tikzpicture}
\node at  (0,0) {\(Q\)};
\node at  (0,0.5) {\(\mathrel{\rotatebox[origin=c]{90}{$\hookrightarrow$}}\)};
\node  at  (0,1) {\(Q'\)};

\node at  (2,0) {\(\Reach(Q)\)};
\node at  (2,0.5) {\(\uparrow\)};
\node at  (2,1) {\(\Reach(Q')\)};

\node at  (4,0) {\(\cR(Q)\)};
\node at  (4,0.5) {\(\uparrow\)};
\node at  (4,1) {\(\cR(Q')\)};
	
\draw[shorten >= 29pt,shorten <= 8pt,->] (0,0) to (2,0);
\draw[shorten >= 29pt,shorten <= 8pt,->] (0,1) to (2,1);
\draw[shorten >= 14pt,shorten <= 28pt,->] (2,0) to (4,0);
\draw[shorten >= 15pt,shorten <= 29pt,->] (2,1) to (4,1);
\end{tikzpicture}
\end{center}
is not functorial at the level of the respective category algebras. Note that depending on the quiver morphism \(Q \hookrightarrow Q'\), even \(\mathbb{K}\cR(Q) \ra \mathbb{K}\cR(Q')\) might not be a homomorphism of algebras; an easy example is induced by the inclusion of an edge into a pair of reciprocal edges. This also shows that if we are to use condensation type operations, such as mapping into \(\cR(Q)\), then functorial persistence is not restored even by restricting to the category of quivers whose morphisms are injective maps on vertices. Nevertheless, as \(\Reach_Q = \cR(Q)\) for acyclic quivers, we still get a fully functorial persistence pipeline by restricting reachability to the category of acyclic quivers.

Despite the discussed shortcoming on the functoriality of a Hochschild homology persistent pipeline for general quivers, we wish to point out here that, in applications to topological data analysis, the above Morita equivalence proves actually to be effective. In fact, by Morita invariance, we have the equality 
$$\beta^\mathrm{HH}_*(\mathbb{K}\Reach_Q) = \beta^\mathrm{HH}_*(\mathbb{K}\cR(Q))$$ of Hochschild Betti numbers, and, from a computational and practical point of view, it is the evolution of the Betti numbers that is of interest in persistent homology. Despite a non-functorial pipeline of Hochschild homology groups of algebras for $\Quiver$, we still obtain a Betti curve, i.e.~the Betti numbers as a function of the filtration parameter, as the following composition:
\begin{equation*}\label{eq:pers_betti_HH}
  (\mathbb{R}, \leq) \ra \Quiver \xrightarrow{\Reach} \Preord\xrightarrow{L} \Poset \xrightarrow{\mathbb{K}\text{-}}\mathbb{K}\text{-}\mathbf{Alg}\xrightarrow{\beta^\mathrm{HH}} \mathbb{N} \ .
\end{equation*}
Note that, applying the same methods as of~\cite{vertechi}, the composition is stable when restricted to acyclic quivers. We summarise the discussion in the following definition: 

\begin{definition}
Let $\mathcal{F}\colon (\mathbb{R},\leq ) \to \Quiver_0$ be a filtration of finite acyclic quivers. Then, its persistent Hochschild homology groups are given by the composition
\begin{equation*}
  (\mathbb{R}, \leq) \ra \Quiver \xrightarrow{\Reach} \Preord\xrightarrow{L} \Poset \xrightarrow{\mathbb{K}\text{-}}\mathbb{K}\text{-}\mathbf{Alg} \ .
\end{equation*} 
If $\mathcal{F}\colon (\mathbb{R},\leq ) \to \Quiver$ is a filtration of  quivers, we define the persistent $\mathrm{HH}$-curves as the Betti curves of the persistent Hochschild homology groups.
\end{definition}

The classical result by Gerstenhaber and Schack \cite{GERSTENHABER1983143} gives a topological interpretation to persistent Hochschild homology. The categorical equivalence \(\Reach_Q \simeq \cR(Q)\) entails the homotopy equivalence \(|\Nerve(\Reach_Q)| \simeq |\Nerve(\cR(Q))|\) of the respective nerves,  hence isomorphism in homology. Then, in view of \cite{GERSTENHABER1983143}, the homology of the nerve \(\Nerve(\cR(Q))\) (when working over a field~\(\mathbb{K}\)) is isomorphic to \(\mathrm{HH}_*(\mathbb{K}\cR(Q))\). Therefore persistent Hochschild homology and persistent homology of \(\Reach_Q\) yield the same Betti curves:

\begin{proposition}
Let $\mathcal{F}\colon (\mathbb{R},\leq ) \to \Quiver$ be a filtration of  quivers. Then, the persistent \(\mathrm{HH}\)-curves agree with the simplicial Betti curves of the nerves of the reachability categories.
\end{proposition}

We hope that this variation of persistent homology of graphs, using reachability categories instead of the more classical Vietoris-Rips complexes, might find concrete applications in topological data analysis. Note that, as mentioned in Example \ref{ex:face_poset_and_Reach}, Hochschild homology of reachability categories can be of arbitrarily high degree, yielding non-trivial information also in degrees $>1$ (cf.~Theorem~\ref{thmhapp}). 

In this last section, we have discussed the use of reachability categories for constructing persistent Hochschild homology pipelines in topological data analysis; this was one of the the main motivations to develop the theory of reachability categories. We have observed that taking the filtration given by the reachability categories, and then applying the standard persistent methods, yields a well-behaved persistent homology theory which we might call \emph{reachability persistent homology}. We conclude with a further perspective on it.

\begin{remark}
    Taking the nerve of categories is a functorial operation, hence the classical persistent homology of \(\Reach_Q\)  yields a functorial persistent homology theory of quivers. Due to a lack of functoriality in the more general \(\mathrm{HH}\)-pipeline, having both a functorial and an \emph{efficiently computable} algebraic persistence via Hochschild homology seems to be unattainable with the tools at hand. Homology of \(\Reach_Q\) is then among the closest functorial ``approximations'' of a functorial persistent Hochschild homology of quivers. The non-functoriality discussed in this section also prohibits the pipeline from having a consistent basis over the filtration to produce a persistence diagram, beyond the Betti curve. The main obstacle, as we see it, is the lack of tools to efficiently compute \(\mathrm{HH}\) of some category or algebra induced from a real world quiver data. Therefore, it remains open what a suitable persistent Hochschild homology theory of general quivers should look like.
\end{remark}

\bibliographystyle{alpha}
\bibliography{references}

\thanks{\(\dagger\) Department of Mathematics, University of Bologna, Bologna, Italy -- luigi.caputi@unibo.it} 

\thanks{\(^\ast\) Nordita, Stockholm University, Stockholm, Sweden -- henri.riihimaki@su.se}
\end{document}